\documentclass[]{elsarticle}

    \makeatletter
    \def\ps@pprintTitle{%
      \let\@oddhead\@empty
      \let\@evenhead\@empty
      \def\@oddfoot{\reset@font\hfil\thepage\hfil}
      \let\@evenfoot\@oddfoot
    }
    \makeatother
\usepackage{amsmath}%
\usepackage{amsfonts}%
\usepackage{amssymb}%
\usepackage{graphicx}
\usepackage{tikz}
\newtheorem{theorem}{Theorem}

\newtheorem{algorithm}{Algorithm}

\newtheorem{conjecture}{Conjecture}
\newtheorem{corollary}{Corollary}

\newtheorem{lemma}{Lemma}

\newtheorem{proposition}{Proposition}
\newtheorem{remark}{Remark}

\numberwithin{equation}{section}

\newproof{proof}{Proof}
\begin{document}

\title{Internal Partitions of Regular Graphs}
\author{Amir Ban}
\ead{amirban@netvision.net.il}
\address{Center for the Study of Rationality, Hebrew University, Jerusalem, Israel}
\author{Nati Linial}
\ead{nati@cs.huji.ac.il} 
\address{School of Computer Science and Engineering, Hebrew University, Jerusalem, Israel}
%\thanks{\textit{Amir Ban} \texttt{amirban@netvision.net.il} Center for the Study of Rationality, Hebrew University, Jerusalem, Israel}
%\thanks{\textit{Nati Linial} \texttt{nati@cs.huji.ac.il} School of Computer Science and Engineering, Hebrew University, Jerusalem, Israel}
\date{}

\begin{abstract}

An {\em internal partition} of an $n$-vertex graph $G=(V,E)$ is a partition of $V$ such that every vertex has at least as many neighbors in its own part as in the other part. It has been conjectured that every $d$-regular graph with $n>N(d)$ vertices has an internal partition. Here we prove this for $d=6$. The case $d=n-4$ is of particular interest and leads to interesting new open problems on cubic graphs. We also provide new lower bounds on $N(d)$ and find new families of graphs with no internal partitions. Weighted versions of these problems are considered as well.

\end{abstract}

\maketitle

\section{Introduction}

It is well-known that every finite graph $G=(V,E)$ has an {\em external partition}, i.e., a splitting of $V$ into two parts such that each vertex has at least half of its neighbors in the other part. This is, e.g., true for $G$'s max-cut partition. Much less is known about the {\em internal partition} problem in which $V$ is split into two non-empty parts, such that each vertex has at least half of its neighbors in its own part. Not all graphs have an internal partition and their existence is proved only for certain classes of graphs. Several investigators have raised the conjecture that for every $d$ there is an $n_0$ such that every $d$-regular graph with at least $n_0$ vertices has an internal partition. Here we prove the case $d=6$ of this conjecture.

A related intriguing concept in this area is the notion of {\em external bisection}. This is an external partition in which the two parts have the same cardinality. We conjecture that the Petersen graph is the only connected cubic graph with no external bisection. We take some steps in resolving this problem.

These concepts have emerged in several different areas and as a result there is an abundance of terminologies here. Thus Gerber and Kobler\cite{Gerber} used the term {\em satisfactory partition} for internal partitions. Internal/external partitions are called {\em friendly} and {\em unfriendly} partitions sometimes. Morris\cite{Morris} studied social learning, and considered a more general problem. Now we want to partition $V=A \dot\cup B$ with $A, B \neq \emptyset$ such that every $x\in A$ (resp $y \in B$) has at least $qd(x)$ of its neighbors in $A$ (resp. $\ge(1-q)d(y)$  neighbors in $B$). He refers to such sets as {\em ($q$/$1-q$)-cohesive}. Here we use the term {\em $q$-internal partitions}. The complementary notion of {\em $q$-external partitions} is considered as well.

\tikzstyle{gray}=[circle, draw, fill=gray!50, inner sep=0pt, minimum width=6pt]
\tikzstyle{green}=[circle, draw, fill=green!50, inner sep=0pt, minimum width=6pt]
\tikzstyle{orange}=[circle, draw, fill=orange!50, inner sep=0pt, minimum width=6pt]

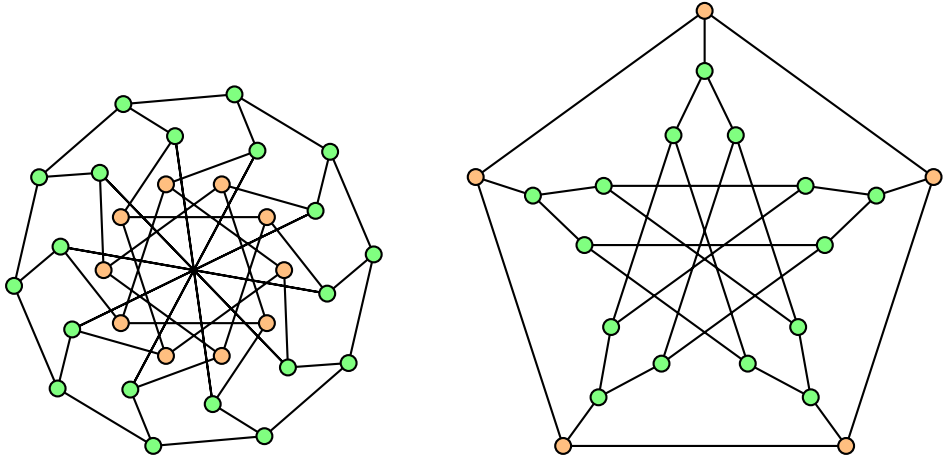
\begin{figure}[tbp]

\begin{minipage}[t]{0.5\textwidth}
%  Tutte's 8-cage
\begin{tikzpicture}[thick,scale=0.6]
    % The following path utilizes several useful tricks and features:
    % 1) The foreach statement is put inside a path, so all the edges
    %    will in fact be a the same path.
    % 2) The node construct is used to draw the nodes. Nodes are special
    %    in the way that they are drawn *after* the path is drawn. This
    %    is very useful in this case because the nodes will be drawn on
    %    top of the path and therefore hide all edge joins.
    % 3) Simple arithmetics can be used when specifying coordinates.
    \draw \foreach \x in {0,36,...,324}
    {
        (\x:2) node [orange] {}  -- (\x+108:2)
        (\x-10:3) node [green] {} -- (\x+5:4)
        (\x-10:3) -- (\x+36:2)
        (\x-10:3) --(\x+170:3)
        (\x+5:4) node [green] {} -- (\x+41:4)
    };

\end{tikzpicture}\quad

\end{minipage}
\begin{minipage}[t]{0.5\textwidth}

%
% The largest 3-regular graph of diameter 3
\begin{tikzpicture}[thick,scale=0.8]%
    \draw \foreach \x in {18,90,...,306} {
        (\x:4) node [orange] {} -- (\x+72:4)
        (\x:4) -- (\x:3) node [green] {}
        (\x:3) -- (\x+15:2) node [green] {}
        (\x:3) -- (\x-15:2) node [green] {}
        (\x+15:2) -- (\x+144-15:2)
        (\x-15:2) -- (\x+144+15:2)
};
\end{tikzpicture}

\end{minipage}

\caption[Examples of internal partitions]{Examples of internal partitions}
\label{cubic-examples}
\end{figure}

Figure \ref{cubic-examples} shows examples of internal partitions of regular cubic graphs.

Bazgan, Tuza and Vanderpooten have written several papers~\cite{Bazgan2003,Bazgan2006} on internal partitions. In~\cite{Bazgan2010} they give a survey of this area. Much of their work concerns the complexity of finding such partitions, a problem which we do not address here.

Our own interest in this subject arose in our studies of learning in social or geographical networks. Vertices in these graphs represent individuals and edges stand for social connection or geographical proximity. The individuals adopt one of two choices of a social attribute (e.g. PC or Mac user). Society evolves over time, with each individual adopting the choice of the majority of her neighbors. We asked whether a stable, diverse assignment of choices is possible in such a society. This amounts to finding an internal partition if the social choices are equally persuasive. It is also of interest to consider the problem when choices carry different persuasive power (say a neighbor who is a Mac user is more persuasive than a PC neighbor). If the merits are in proportion $q : 1-q$, this leads to the problem of finding a {\em $q$-internal partition}.

Thomassen \cite{Thomassen} showed that for every two integers $s, t > 0$ there is a $g=g(s,t)$ such that every graph $G=(V,E)$ of minimum degree at least $g$ has a partition $V=V_1\dot\cup V_2$ so that the induced subgraphs $G(V_1), G(V_2)$ have minimum degree at least $s, t$, respectively. He conjectured that the same holds with $g(s,t) = s + t + 1$, which would be tight for complete graphs. Stiebitz \cite{Stiebitz} proved this conjecture, and extended it as follows: For every $a, b: V \mapsto \mathbb{Z}_{+}$ such that $\forall v \in V, d_G(v) \geq a(v) + b(v) + 1$, there exists a partition of $V=A\dot\cup B$, such that $\forall v \in A, d_A(v) \geq a(v)$ and $\forall v \in B, d_B(v) \geq b(v)$.  Kaneko~\cite{Kaneko} showed that in triangle-free graphs the same conclusion holds under the weaker assumption $d_G(v) \geq a(v) + b(v)$.

Stiebitz's result shows that, given $q \in (0,1)$, every graph has a non-trivial partition which is at most one edge (for each vertex) short of being a $q$-internal partition. Shafique and Dutton~\cite{Shafique} showed the existence of internal partitions in all cubic graphs except $K_4$ and $K_{3,3}$ and in all 4-regular graphs except $K_5$. In this paper, we settle the problem for 6-regular graphs.

Shafique and Dutton also conjectured that $K_{2k+1}$ is the only $d=2k$-regular graph with no internal partition. We disprove this and present a number of counterexamples. Many of these exceptions are with $d \geq n-4$. This range turns out to be of interest and we discuss it as well. As we show, there exist $d$-regular $n$-vertex graphs with no internal partitions with both $d$ and $n-d$ arbitrarily large. We conjecture that every $2k$-regular graph with $n \geq 4k$ has an internal partition. In the process, we consider external bisections of regular graphs, and especially cubic graphs. We note that all class-I cubic graphs have an external bisection, and speculate that for class-II cubic graphs, only graphs that have the Petersen graph as a component do {\em not} have such a bisection.

Finally, we conjecture that there is a function $\mu=\mu(d,q)$ such that if $qd$ is an integer, then every $d$-regular graph has a $q$-internal partition. We also conjecture this for $q = 1/2$ and $d$ odd. As we show, for $d$ fixed and large $n$, every $n$-vertex $d$-regular graph has {\sl many} $q$-internal partitions for {\em some} $q$. This lends some support to our conjecture. We also discuss an algorithm that generates $q$-internal partitions of a graph for many, and plausibly all values of $q$. This sheds light on what causes a graph to be non-partitionable.

\section{Terminology}
We consider undirected graphs $G=(V,E)$ with $n$ vertices. For $S \subset V$, we denote by $G(S)$ the induced subgraph of $S$. The degree of $x\in V$ is denoted by $d(v)=d_G(v)$ and the number of neighbors that $v$ has in $S\subseteq V$ is called $d_S(v)$. The complement of $G$ is denoted by $\bar{G}$.

 A {\em bisection} of $V=A\dot\cup B$ is a partition with $|A| = |B|$. If $||A| - |B|| \leq 1$, then we call it a near-bisection. Corresponding to the partition $(A,B)$ of $V$ is the {\em cut} $E(A,B)=E_G(A,B)=\{xy\in E|x\in A, y\in B\}$. For $x \in A$ and $y \in B$ we call $d_A(x), d_B(y)$, respectively, the vertices' {\em indegrees}, and $d_B(x), d_A(y)$ the {\em outdegrees}. These terms usually refer to directed graphs, but we could not resist the convenience of using them in the present context.

A subset $S \subseteq V$ is called {\em $p$-cohesive} if $\forall x \in S, d_S(x) \geq p$. It is called a {\em $p$-crumble} if no $S' \subseteq S$ is $p$-cohesive. (Note that our notion of cohesion differs from that of Morris~\cite{Morris}).

A partition $(A,B)$ is {\em $q$-internal} for $q \in (0,1)$ if $\forall x \in A, d_A(x) \geq qd_G$ and  $\forall x \in B, d_B(x) \geq (1-q)d_G(x)$. A $\frac{1}{2}$-internal partition is simply {\em internal}. 

If $\forall x \in A, d_B(x) \geq qd_G$ and  $\forall x \in B, d_A(x) \geq (1-q)d_G(x)$ we call the partition {\em $q$-external}. A $\frac{1}{2}$-external partition is {\em external}.

A $q$-internal or a $q$-external partition is called {\em integral} if for every $v \in V$, $qd_G(v)$ is an integer.

A $q$-internal or a $q$-external partition $(A,B)$ is called {\em exact} if $|A| = qn$, and {\em near-exact} if $||A| - qn| < 1$. A $\frac{1}{2}$-exact partition is a {\em bisection}. For $q = \frac{1}{2}$, near-exact partitions are {\em near-bisections}.

\section{Internal Partitions of 6-Regular Graphs}

\begin{lemma}
\label{l}
Let $G=(V,E)$ be a graph with minimal degree $d$. For $0 < k < |V|$, let $(A,B)$ be a partition of $V$ that attains $\min |E(A,B)|$ over all partitions with $|A|=k$ or $|B|=k$. Then, either:
\begin{enumerate}
\item \label{l1} $A$ is $l$-cohesive and $B$ is $m$-cohesive for some integers $l,m$ with $l + m = d$, or:
\item \begin{enumerate}
   \item \label{l2a} $A$ is $l$-cohesive and $B$ is $m$-cohesive for some integers $l,m$ with $l + m = d - 1$, and:
   \item \label{l2b} The vertices in $A$ with indegree $l$ and the vertices in $B$ with indegree $m$ form a complete bipartite subgraph in $G$, and:
   \item \label{l2c} For every $x \in A$ with indegree $l$, $B \cup \{x\}$ is $(m+1)$-cohesive. Similarly, $A \cup \{x\}$ is $(l+1)$-cohesive for every $x \in B$ with indegree $m$.
   \end{enumerate}
\end{enumerate}
\end{lemma}

\begin{proof}
Let $x \in A, y \in B$. If $xy \notin E$ then
\begin{align*}
|E(((A \backslash \{x\}) \cup \{y\}, (B \backslash \{y\}) \cup \{x\})| - |E(A,B)| & = \\
 =d_A(x) - d_B(x) + d_B(y) - d_A(y) \leq &\\
\leq 2[d_A(x) + d_B(y) - d]
\end{align*}
If $xy\in E$, then
\begin{align*}
|E(((A \backslash \{x\}) \cup \{y\}, (B \backslash \{y\}) \cup \{x\})| - |E(A,B)| & = \\
= d_A(x) - d_B(x) + (d_B(y) + 1) - (d_A(y) - 1) \leq &\\
\leq 2[d_A(x) + d_B(y) - (d - 1)]
\end{align*}
Since $E(A,B)$ is minimal, it follows that the sum of indegrees is at least $d-1$ if $x, y$ are adjacent, and $d$ otherwise.

Let us apply this for $x,y$ of minimum indegree. Then (\ref{l1}) follows if there is such a pair with $xy \notin E$. On the other hand, if $xy\in E$ for all such pairs, then (\ref{l2a}) and (\ref{l2b}) follow. We obtain (\ref{l2c}) by observing that increasing by one the indegree of all minimum indegree vertices in a subset, increases the minimum indegree of the subset by one.
\qed
\end{proof}

\begin{corollary}
Every $n$-vertex $d$-regular graph has a $\lceil \frac{d}{2} \rceil$-cohesive set of at most $\lceil \frac{n}{2} \rceil$ vertices (resp. $\frac{n}{2} + 1$) for $d$ even (for $d$ odd).
\end{corollary}

\begin{proof}
Consider a near-bisection of $G$ that minimizes $|E(A,B)|$. By Lemma \ref{l} if $d$ is even, at least one of $A, B$ is $\frac{d}{2}$-cohesive. If $d$ is odd, and if neither $A$ nor $B$ are $\lceil \frac{d}{2} \rceil$-cohesive, then by (\ref{l2a}) both are $\lfloor \frac{d}{2} \rfloor$-cohesive, and by (\ref{l2c}) each can be made $\lceil \frac{d}{2} \rceil$-cohesive by adding a vertex of the other.
\qed
\end{proof}

\begin{theorem}
\label{6regular}
Every $6$-regular graph with at least $14$ vertices has an internal partition.
\end{theorem}

\begin{proof}
We argue by contradiction and consider an $n$-vertex $6$-regular graph $G=(V,E)$ with no internal partition. Let $(A, B)$ be the near-bisection of $V$ that attains $\min|E(A,B)|$ over all near-bisections.
By Lemma \ref{l} either $A$ or $B$ must be 3-cohesive. We may assume $A$ is 3-cohesive while $B$ is not, for else $(A,B)$ is an internal partition.

We repeatedly carry out the following step:
As long as there is some $y \in B$ with outdegree $d_A(y)> 3$ we move that vertex from $B$ to $A$. If $A$ is 3-cohesive then clearly so is $A\cup \{y\}$, while if $B$ is 3-crumble, so is $B \backslash \{y\}$. By assumption no internal partition exists, so this process must terminate with a trivial partition, i.e., $B$ must be 3-crumble. The move of $y$ from $B$ to $A$ decreases $|E(A,B)|$ by $2d_A(y)-6 \geq 2$. Every step of the process therefore decreases the cut by at least 2, while $|B|$ decreases by 1. Also in the last two moves $|E(A,B)|$ decreases by $\ge 4$, and $6$ in this order, and at termination $E(A,B)=\emptyset$. We conclude that $|E(A,B)| \geq 2|B| + 6$.

On the other hand $|E(A,B)| \leq 2|A| + 4$:
By Lemma \ref{l} all vertices in $A$ have outdegree $\le 2$, except for at most 4 (that  are adjacent to a vertex in $B$ with outdegree $\le 4$) vertices with outdegree 3. Therefore $2|A| + 4 \ge |E(A,B)| \geq 2|B| + 6$ so that $|A| \ge |B|+1$.
It follows that $|A| = |B| + 1$, $n$ is odd and $B$ is a ``tight'' 3-crumble. Namely, exactly 4 vertices in $A$ have outdegree 3, and in all moves (except the last two) $|E(A,B)|$ is reduced by exactly 2. If $n \geq 9$ then $|B| \geq 4$, so the first two vertex moves are of outdegree 4. Let $y', y'' \in B$ be these first two vertices, let $(A',B') = (A \cup \{y'\},B \backslash \{y'\})$ be the partition after the first move, and let $(A'',B'') = (A \cup \{y',y''\},B \backslash \{y',y''\})$ be the partition after the second move. By the above $|E(A',B')| = |E(A,B)| - 2$ and $|E(A'',B'')| = |E(A,B)| - 4$.

By Lemma \ref{l} (\ref{l2c}) all vertices in $A'$ have outdegree 2. Therefore, in $A''$, all vertices have outdegree 2 except 4 with outdegree 1. Suppose that some pair of these outdegree-2 vertices in $A''$, say $x', x''$ are adjacent. Then it would be possible to move both vertices to $B''$ while increasing the cut size by only 3. Namely, $|E(A'' \backslash \{x',x''\},B'' \cup \{x',x''\})| = |E(A'',B'')|  + 3 < |E(A,B)|$.  This yields a near-bisection, that contradicts the minimality of $|E(A,B)|$. Alternatively, if the outdegree-2 vertices in $A''$ form an independent set, then all their neighbors in $A''$ must have outdegree 1 and indegree 5. It follows that there are at most 5 vertices in $A''$ of outdegree-2. Therefore $|A''| \leq 9 \Rightarrow |A| \leq 7 \Rightarrow n \leq 13$.
\qed
\end{proof}

\begin{remark}
We now comment on the range $n \le 13$. Note that the proof covers all even $n$. The complete graph $K_7$ is an exception with $n=7$.

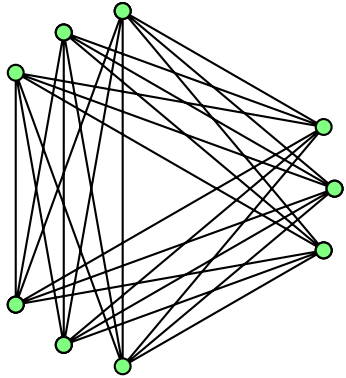
\begin{figure}[tbp]
\centering

%  Tutte's 8-cage
\begin{tikzpicture}[thick,scale=0.6]
    % The following path utilizes several useful tricks and features:
    % 1) The foreach statement is put inside a path, so all the edges
    %    will in fact be a the same path.
    % 2) The node construct is used to draw the nodes. Nodes are special
    %    in the way that they are drawn *after* the path is drawn. This
    %    is very useful in this case because the nodes will be drawn on
    %    top of the path and therefore hide all edge joins.
    % 3) Simple arithmetics can be used when specifying coordinates.
    \draw \foreach \x in {0,120,240}
    {
        (\x:4) node [green] {}  -- (\x+100:4)
        (\x-20:4) node [green] {} -- (\x+100:4)
        (\x+20:4) node [red] {} -- (\x+100:4)
        (\x:4) node [green] {}  -- (\x+120:4)
        (\x-20:4) node [green] {} -- (\x+120:4)
        (\x+20:4) node [red] {} -- (\x+120:4)
        (\x:4) node [green] {}  -- (\x+140:4)
        (\x-20:4) node [green] {} -- (\x+140:4)
        (\x+20:4) node [green] {} -- (\x+140:4)
    };

\end{tikzpicture}\quad

\caption[$K_{3,3,3}$: A 6-regular graph with no internal partition]{$K_{3,3,3}$: A 6-regular graph with no internal partition}
\label{d=n-3 example}
\end{figure}

For $n=9$, there is a unique unpartitionable 6-regular graph (see Figure \ref{d=n-3 example}). We prove this statement when we discuss the case $d = n -3$ in the following section.

For $n=11$, there exist 6-regular graphs with no internal partition. One such example, $Q_3$, is a member of a class of unpartitionable graphs we construct in Section \ref{general_case}.

The case $n=13$ remains unsettled. Our Conjecture \ref{conj2d} would imply that all such graphs have an internal partition.
\end{remark}

\section{Partitions of Complementary Graphs}

\begin{proposition}
\label{coexist}
For every $q \in (0,1)$, every graph $G$ has a $q$-external partition.
\end{proposition}

\begin{proof}
For a partition $(A,B)$ define
\begin{equation}
w(A,B) := |E(A,B)| - (1-q)\sum\limits_{x \in A} d_G(x) - q\sum\limits_{x \in B} d_G(x)
\end{equation}

The partition that maximizes $w(A,B)$ is non-trivial, since for every non-isolated vertex $x$ there holds $w(V \backslash \{x\},  \{x\}) > w(V,\emptyset)$ and $w(\{x\}, V \backslash \{x\}) > w(\emptyset,V)$. Furthermore $w(A,B) - w(A \backslash \{x\}, B \cup \{x\}) = d_B(x) - d_A(x) + (1-q) d_G(x) - qd_G(x) = 2d_B(x) - 2qd_G(x)$ and $w(A,B) - w(A \cup \{x\}, B \backslash \{x\}) = d_A(x) - d_B(x) -(1- q) d_G(x) + qd_G(x) = 2d_A(x) - 2(1-q)d_G(x)$,  so the maximality of $(A,B)$ implies that it is $q$-external.
\qed
\end{proof}

\begin{proposition}
\label{dual}
For $q \in (0,1)$ every exact $q$-internal partition of $G=(V,E)$ is an exact $(1-q)$-external partition of $\bar{G}$.
\end{proposition}

\begin{proof}
Let $|V|=n$ and let $(A,B)$ be an exact $q$-internal partition of $G$. Namely, $|A| = qn, |B|=(1-q)n$ and $\forall x \in A, d_A(x) \geq qd_G(x)$ and $\forall x \in B, d_B(x) \geq (1-q)d_G(x)$. To indicate that we work in $\bar{G}$ we denote by $\bar{A}, \bar{B}$ the subgraphs of $\bar{G}$ induced by $A, B$. Then:
\begin{align*}
&\forall x \in V,&d_{\bar{G}}(x) = n - d_G(x) - 1 \\
&\forall x \in A,&d_{\bar{B}}(x) = |B| - d_B(x) = (1-q)n - (d_G(x) - d_A(x)) \geq \\
&&\geq (1-q)(n - d_G(x)) > (1-q)d_{\bar{G}}(x) \\
&\forall x \in B,&d_{\bar{A}}(x) = |A| - d_A(x) = qn - (d_G(x) - d_B(x)) \geq \\
&& \geq q(n - d_G(x)) > qd_{\bar{G}}(x)
\end{align*}

So $(A,B)$ is a $(1-q)$-external partition.
\qed
\end{proof}

\begin{proposition}
For $q \in (0,1)$ every exact $(1-q)$-external partition of $G=(V,E)$  is an exact $q$-internal partition of $\bar{G}$, provided the partition of $\bar{G}$ is integral.
\end{proposition}

\begin{proof}
Maintaining the notation of Proposition \ref{dual}, consider an exact $(1-q)$-external partition $(A,B)$ of $G$. Namely $|A| = qn, |B|=(1-q)n$ and $\forall x \in B, d_A(x) \geq qd_G(x)$ and $\forall x \in A, d_B(x) \geq (1-q)d_G(x)$. Then:
\begin{align*}
&\forall x \in V,&d_{\bar{G}}(x) = n - d_G(x) - 1 \\
&\forall x \in A,&d_{\bar{A}}(x) = |A| - d_A(x) - 1 = qn - (d_G(x) - d_B(x)) - 1 \geq \\
&&\geq q(n - d_G(x)) - 1 = qd_{\bar{G}}(x) - (1-q).
\end{align*}
By rounding up we conclude that $d_{\bar{A}}(x) \geq qd_{\bar{G}}(x)$. (Note that $d_{\bar{A}}(x)$ and $qd_{\bar{G}}(x)$ are integers and $1>q>0$).
\begin{align*}
&\forall x \in B,&d_{\bar{B}}(x) = |B| - d_B(x) = (1-q)n - (d_G(x) - d_A(x)) - 1 \geq \\
&& \geq (1-q)(n - d_G(x)) - 1 = (1-q)d_{\bar{G}}(x) - q.
\end{align*}
By a similar argument $d_{\bar{B}}(x) \geq (1-q)d_{\bar{G}}(x)$,
so $(A,B)$ is a $q$-internal partition.
\qed
\end{proof}

\begin{corollary}
If $G$ has an internal bisection, then $\bar{G}$ has an external bisection.
\end{corollary}

\begin{corollary}
\label{dual-bisection}
If all degrees in $G$ are even and $\bar{G}$ has an external bisection, then $G$ has an internal bisection.
\end{corollary}

\begin{theorem}
For even $n$, every ($n-2$)-regular graph has an internal bisection.
\end{theorem}

\begin{proof}
The complement of an ($n-2$)-regular graph is a perfect matching. Split each matched pair between sides of a partition to obtain an external bisection. The theorem follows from Corollary \ref{dual-bisection}.
\qed
\end{proof}

\begin{theorem}
An ($n-3$)-regular graph $G$ has an internal partition if and only if its complementary graph $\bar G$ has at most one odd cycle. Furthermore this partition is a near-bisection.
\end{theorem}

\begin{proof}
Clearly $\bar{G}$ is 2-regular, i.e. it is comprised of vertex disjoint cycles. For every cycle, place the vertices alternately in $A$ and in $B$. If at most one cycle is odd, then $||A| - |B|| \leq 1$, so the partition is a near-bisection. It is also an internal partition of $G$, since the smaller side, say $B$, is a clique. Also, $A$ spans a clique if $|A| = |B|$ , or a clique minus one edge if $|A| = |B| + 1$, so its minimum indegree is also $|B| - 1$. As $|B| - 1 \geq (n-3)/2$, the partition is internal.

Let $G$ have an internal partition $(A,B)$. If $n$ is even, every vertex must have indegree $\geq n/2 - 1$. Therefore $|A| = |B| = n/2$ and the complementary graph $\bar{G}$ is bipartite so has no odd cycles. If $n$ is odd, assume $|A| > |B|$. $B$'s minimum indegree is $(n-3)/2$ so $|B| = (n-1)/2, |A| = (n+1)/2$ and the partition is a near-bisection. In $\bar{G}$, $|E(A,B)|=2|B|=n-1$ so $E(A)=(2|A|-|E(A,B)|)/2 = 1$. Therefore $(A,B)$ is bipartite in $\bar{G}$ except for a single edge internal to $A$. Therefore $\bar{G}$ has only one odd cycle.
\qed
\end{proof}

We can now confirm that $K_{3,3,3}$, the graph in Figure \ref{d=n-3 example}, has no internal partition, as it is the complement of three disjoint triangles. Furthermore, as there is no other way for a 9-vertex graph to have more than one odd cycle, this is the only $n=9, d=6$ graph with this property.

\section{The Case $d = n - 4$ and Cubic Graphs}

Let $G$ be a $d$-regular graph on $n$ vertices with $d = n - 4$. Clearly $n$ must be even, and its complement $\bar{G}$ is a cubic graph.

\begin{proposition}
\label{n-4}
If an ($n-4$)-regular graph $G$ has an internal partition then either
\begin{itemize}
\item $\bar{G}$ has an external bisection, or
\item $\bar{G}$ has an independent set of size at least $n/2-1$.
\end{itemize}
\end{proposition}

\begin{proof}
By Corollary \ref{dual-bisection} if $\bar{G}$ has an external bisection, $G$ has an internal bisection. If not, to be internal a partition must have minimum degree $n/2-2$ so each part must have size $\ge n/2-1$. Therefore $|A| = |B|+2$, where $B$ is a clique in $G$ and an anticlique in $\bar{G}$.
\qed
\end{proof}
 
\begin{figure}[tbp]

\begin{center}
\begin{tikzpicture}[style=thick]
\draw (18:2cm) -- (90:2cm) -- (162:2cm) -- (234:2cm) --
(306:2cm) -- cycle;
\draw (18:1cm) -- (162:1cm) -- (306:1cm) -- (90:1cm) --
(234:1cm) -- cycle;
\foreach \x in {18,90,162,234,306}{
\draw (\x:1cm) -- (\x:2cm);
\draw (\x:2cm) [green] circle (3pt);
\draw (\x:1cm) [green] circle (3pt);
}
\draw (18:2cm) [orange] circle (3pt);
\draw (162:2cm) [orange] circle (3pt);
\draw (234:1cm) [orange] circle (3pt);
\draw (306:1cm) [orange] circle (3pt);
\end{tikzpicture}
\end{center}
\caption[External partition of the Petersen graph]{External partition of the Petersen graph}
\label{petersen}
\end{figure}
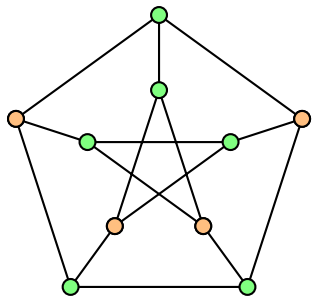

The Petersen graph (see Figure \ref{petersen}) has no external bisection, but it has an independent set of size 4. Its complement is 6-regular, and in fact has an internal partition (but not a bisection), as already proved in Theorem \ref{6regular}.

The requirement of an independent set of size $n/2-1$ means that, save for 3 edges, the cubic graph is bipartite. Clearly this is a rare phenomenon among cubic graphs, so our quest for graphs with internal partitions boils down to asking which cubic graphs have an external bisection.

We show next:

\begin{theorem}
\label{class-1}
Every class-1 3- or 4-regular graph $G$ has an external bisection.
\end{theorem}

\begin{proof}
Pick some $d$-edge coloring of $G$, and choose any two of the colors. The corresponding alternating cycles form a $2$-factor in $G$ of even cycles. Number the vertices of each of these cycles sequentially along the cycle path. Alternately assign the vertices in the cycles to the two sides of a partition which is clearly a bisection. For $d \leq 4$, this partition is external, since every vertex has at least two neighbors at the opposite part.
\qed
\end{proof}

While all class-1 cubic graphs have an external bisection, the same question for class-2 cubic graphs remains open, though below we present a partial result. As noted, the Petersen graph, the smallest {\em snark}, has no external bisection. We checked a substantial number of larger snarks and found external bisections in all of them. Our computer experiments also suggest that all cubic graphs with bridges have external bisections, so we make the conjecture:

\begin{conjecture}
\label{cubic}
The Petersen graph is the only connected cubic graph that has no external bisection.
\end{conjecture}

Note that disconnected cubic graphs with no external bisection do exist. For example, a graph that has an odd number of components that are Petersen graphs and any number of $K_4$ components.

\begin{figure}[tbp]

\begin{minipage}[t]{0.5\textwidth}
\begin{tikzpicture}[style=thick]
\foreach \pos/\name in  {{(18:2cm)/a}, {(90:2cm)/b},  {(162:2cm)/c}, {(234:2cm)/d}, {(306:2cm)/e}}
        \node[green] (\name) at \pos {};
\foreach \pos/\name in  {{(18:1cm)/f}, {(90:1cm)/g},  {(162:1cm)/h}, {(234:1cm)/i}, {(306:1cm)/j}}
        \node[green] (\name) at \pos {};
\draw (a) -- (b) -- (c) -- (d) -- (e) -- (a);
\draw (f) -- (h) -- (j) -- (g) -- (i) -- (f);
\draw (a) -- (f);
\draw (b) -- (g);
\draw (c) -- (h);
\draw (d) -- (i);
\draw (e) -- (j);

\foreach \pos/\name in  {{(-1,-4)/k}, {(-2,-3)/l},  {(2,-3)/m}, {(1,-4)/n}}
        \node[green] (\name) at \pos {};
\draw (k) -- (l) -- (m) -- (n) -- (k);
\draw (k) -- (m);
\draw (l) -- (n);

\end{tikzpicture}\quad

\end{minipage}
\begin{minipage}[t]{0.5\textwidth}

\begin{tikzpicture}[style=thick]
\foreach \pos/\name in  {{(18:2cm)/a}, {(90:2cm)/b},  {(162:2cm)/c}, {(234:2cm)/d}, {(306:2cm)/e}}
        \node[green] (\name) at \pos {};
\foreach \pos/\name in  {{(18:1cm)/f}, {(90:1cm)/g},  {(162:1cm)/h}, {(234:1cm)/i}, {(306:1cm)/j}}
        \node[green] (\name) at \pos {};
\draw (a) -- (c);
\draw (a) -- (d);
\draw (a) -- (g);
\draw (a) -- (h);
\draw (a) -- (i);
\draw (a) -- (j);

\draw (b) -- (d);
\draw (b) -- (e);
\draw (b) -- (f);
\draw (b) -- (h);
\draw (b) -- (i);
\draw (b) -- (j);

\draw (c) -- (e);
\draw (c) -- (f);
\draw (c) -- (g);
\draw (c) -- (i);
\draw (c) -- (j);

\draw (d) -- (f);
\draw (d) -- (g);
\draw (d) -- (h);
\draw (d) -- (j);

\draw (e) -- (f);
\draw (e) -- (g);
\draw (e) -- (h);
\draw (e) -- (i);

\foreach \pos/\name in  {{(-1,-4)/k}, {(-2,-3)/l},  {(2,-3)/m}, {(1,-4)/n}}
        \node[green] (\name) at \pos {};

\foreach \a in {k, l, m, n}
	\foreach \b in {a, b, c, d, e, f, g, h, i, j}
		\draw (\a) -- (\b);

\end{tikzpicture}

\end{minipage}

\caption[Smallest d=n-4 regular graph with no internal partition (right) is complement of cubic graph on left]{Smallest d=n-4 regular graph with no internal partition (right) is complement of cubic graph on left}
\label{petersen-k4}
\end{figure}
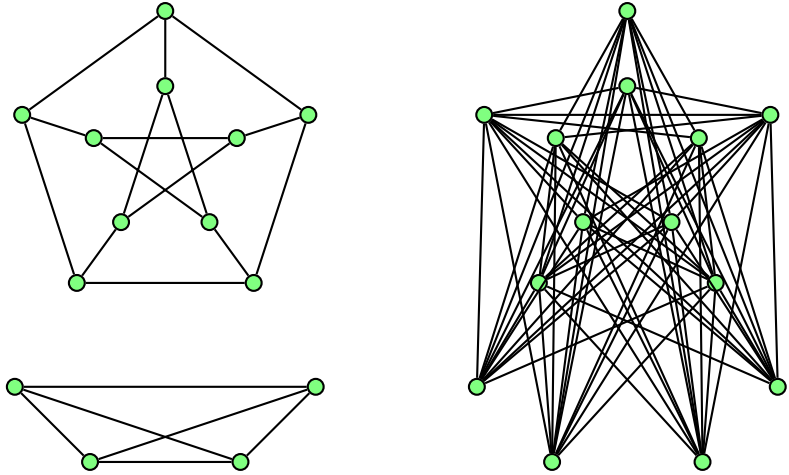
As mentioned above, the complement of the Petersen graph has an internal partition, by virtue of having an anticlique of size $n/2-1$ (as required by Proposition \ref{n-4}). But the above-mentioned disconnected cubic graphs do not meet that requirement and so their complements have no internal partition. The smallest of these is a 10-regular graph of order 14, whose complement is a Petersen graph plus a $K_4$ component (see Figure \ref{petersen-k4}). This is the smallest of an infinite class of $d=(n-4)$-regular graphs with no internal partition. If Conjecture \ref{cubic} is true, these are the only exceptions, as stated in the following:

\begin{conjecture}\label{cnj2}
If $G$ is $(n-4)$-regular and has no internal partition, then $\bar G$ is a disconnected cubic graph that has an odd number of components that are Petersen graphs. All other components of $\bar G$ have the property that all their external partitions are bisections.
\end{conjecture}

Another consequence of Conjecture~\ref{cubic} is:

\begin{conjecture}\label{cnj3}
Every cubic graph has an external partition $(A,B)$ with $||A| - |B|| \leq 2$.
\end{conjecture}

\begin{figure}[tbp]

\centering
\begin{tikzpicture}[style=thick]
\foreach \pos/\name in {{(-4,0)/1}, {(-1,0)/3}, {(1,0)/5}, {(3,0)/7}, {(4,3)/9}, {(2,3)/10}, {(0,3)/12},
			       {(-2,4)/14}, {(-3,1.5)/17}, {(-3,3.5)/19}, {(-1,2)/20}, {(-2,1)/25}, {(0,1)/22}, {(2,1)/27}}
        \node[green] (\name) at \pos {};
\foreach \pos/\name in {{(-2,0)/2}, {(0,0)/4}, {(2,0)/6}, {(4,1)/8}, {(1,3)/11}, {(-1,3)/13}, {(-4,4)/15},
			       {(-3,0.5)/16}, {(-3,2.5)/18}, {(0,2)/21}, {(-2,2)/24}, {(-1,1)/23}, {(1,1)/26}, {(3,2)/28}}
        \node[orange] (\name) at \pos {};

\draw (1) -- (2) -- (3) -- (4) -- (5) -- (6) -- (7) -- (8) -- (9) -- (10) -- (11) -- (12) -- (13) -- (14) -- (15) -- (1);
\draw (1) -- (16) -- (17) -- (18) -- (19) -- (20) -- (21) -- (22) -- (23) -- (24) -- (17);
\draw (5) -- (26) -- (11);
\draw (26) -- (27) -- (8);
\draw (7) -- (28) -- (10);
\draw (28) -- (9);
\draw (6) -- (27);
\draw (2) -- (25) -- (16);
\draw (25) -- (18);
\draw (3) -- (23);
\draw (4) -- (22);
\draw (15) -- (19);
\draw (14) -- (24);
\draw (13) -- (20);
\draw (12) -- (21);
\end{tikzpicture}

\caption[Possibly largest ($n=28$) connected cubic graph with no uneven external partition]{Possibly largest ($n=28$) connected cubic graph with no uneven external partition}
\label{all-bisection}
\end{figure}
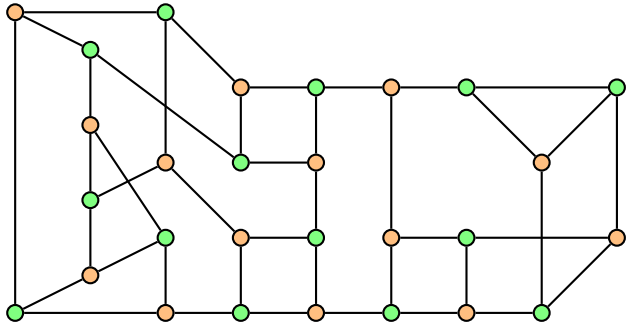

There exist graphs other than $K_4$ all of whose external partitions are bisections. Every cubic graph of order 6 or 8 has this property, since an uneven external partition has at most a $3:2$ proportion of the sides. There are, however, larger connected cubic graphs with this property. The graph in Figure \ref{all-bisection} has order 28 and it may be the largest such graph.

An obvious first step in proving Conjecture~\ref{cubic} would be to show that the smallest counterexample to this conjecture must be bridgeless. We are presently unable to establish even that, but following is a partial result in that direction:

Every bridge in a cubic graph $G = (V,E)$ may be eliminated, resulting in two smaller cubic graphs by the following procedure. The reader may find it useful to follow Figure \ref{bridge} where this procedure is illustrated.

Start by deleting the two vertices of the bridge ($b_1, b_2$). In each of the two components all vertices then have degree $3$, except for two vertices of degree $2$. The following is repeated in a loop for each component until a cubic graph remains:

\begin{itemize}
\item
If the two degree-2 vertices are not adjacent, add an edge between them. This yields a cubic graph, and the procedure is terminated. Otherwise remove them both. The continuation depends on whether the two vertices share a neighbor:
\item
If the removed degree-2 vertices had a common neighbor (such as $p_1, p_2$ and their common neighbor $p_3$), delete that neighbor and  its remaining neighbor (in the example: $p_4$). There remain exactly two vertices of degree 2 ($x_1, y_1$), and the loop is repeated.
\item
Otherwise (as in $q_1, q_2$) their additional neighbors ($q_3, q_4$) are distinct. Again, exactly two vertices with degree 2 remain, and the loop is repeated.
\end{itemize}
The terminal components $G_1 = (V_1,E_1), G_2 = (V_2,E_2)$ are nonempty and cubic, since during the run of the procedure the component always has two vertices of degree 2. 
They each contain a single edge that is not in $E$, namely $x_1 y_1 \in E_1 , x_2 y_2 \in E_2$.

\begin{figure}[tbp]

\centering
\begin{tikzpicture}[style=thick]
\foreach \pos/\name in {{(-5,0)/x_1}, {(-5,2)/y_1}, {(4,0)/x_2}, {(4,2)/y_2}}
        \node[gray] (\name) at \pos {$\name$};
\foreach \pos/\name in {{(-1,1)/b_1}, {(-4,1)/p_4}, {(-2,0)/p_2}, {(2,2)/q_1}, {(3,0)/q_4}}
        \node[green] (\name) at \pos {$\name$};
\foreach \pos/\name in {{(-3,1)/p_3}, {(-2,2)/p_1}, {(1,1)/b_2}, {(2,0)/q_2}, {(3,2)/q_3}}
        \node[orange] (\name) at \pos {$\name$};

\draw (0,1.2) node {$bridge$};
\draw (x_1) -- (p_4) -- (p_3) -- (p_2) -- (b_1) -- (b_2) -- (q_2) -- (q_4) -- (x_2);
\draw (y_1) -- (p_4);
\draw (b_2) -- (q_1) -- (q_3) -- (y_2);
\draw (p_3) -- (p_1) -- (b_1);
\draw (p_1) -- (p_2);
\draw (q_1) -- (q_2);
\draw (q_3) -- (q_4);
\draw [dashed] (x_1) -- (y_1);
\draw [dashed] (x_2) -- (y_2);
\draw (-6,0) -- (x_1);
\draw (-6,-0.5) -- (x_1);
\draw (-6,2) -- (y_1);
\draw (-6,2.5) -- (y_1);
\draw (5,0) -- (x_2);
\draw (5,-0.5) -- (x_2);
\draw (5,2) -- (y_2);
\draw (5,2.5) -- (y_2);

\draw (-6,3.5) node {$G_1$};
\draw [dotted] (-5.5,-1.5) arc (-40:40:4);
\draw (5,3.5) node {$G_2$};
\draw [dotted] (4.5,3.5) arc (140:220:4);
\end{tikzpicture}

\caption[Cubic graph bridge decomposition]{Cubic graph bridge decomposition}
\label{bridge}
\end{figure}
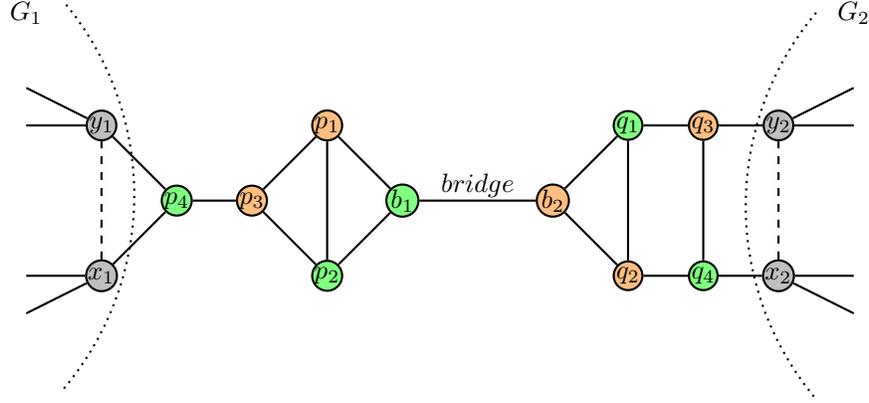

We now note that if $G_1$ and $G_2$ are both class-1, then $G$ has an external bisection, constructed as follows: Bisect the vertices in $V_1$ as in the proof of Theorem \ref{class-1}, taking care to choose the two colors other than $x_1 y_1$'s color. This creates an external bisection of $G_1$ in which $x_1 y_1$ may be removed and replaced by other edges without disturbing the fact that the partition is external. Similarly derive an external bisection of $G_2$, using two colors other than $x_2 y_2$'s color. Finally assign the bridge vertices to different sides of the partition, and do the same with any non-bridge vertex pair that was deleted to obtain $G_1$ and $G_2$. The result is an external bisection of $G$.

Much remains to be done here, since this construction does not work if either $G_1$ or $G_2$ are class-2. It may fail because the graph at hand is a snark that has no 3-edge-coloring, but also if it contains a bridge, due to the requirement pertaining to the color of the non-$E$ edge. If there is more than one such edge, it is not necessarily the case that we can simultaneously satisfy more than one such requirement. 

\section{The General Case}
\label{general_case}

The existence of internal partitions for $d$-regular graphs with $d=5$ and with $7 \leq d \leq n-5$ remains unsettled, as is the existence of $q$-internal partitions for $q \neq \frac{1}{2}$.

\begin{figure}[tbp]

\begin{minipage}[t]{0.5\textwidth}
\begin{tikzpicture}[style=thick]
\foreach \pos/\name in  {{(18:1.5cm)/a}, {(90:1.5cm)/b},  {(162:1.5cm)/c}, {(234:1.5cm)/d}, {(306:1.5cm)/e}}
        \node[green] (\name) at \pos {};
\draw (a) -- (b) -- (c) -- (d) -- (e) -- (a);
\draw [dotted] (0,0) ellipse (1.7cm and 1.7cm);
\draw (-0.8,1.2) node {$X_2$};

\foreach \pos/\name in  {{(-1,-4)/n}, {(0,-2.5)/l}, {(1,-4)/m}}
        \node[green] (\name) at \pos {};
\draw (n) -- (l) -- (m) -- (n);
\draw [dotted] (0,-3.3) ellipse (1.5cm and 1.5cm);
\draw (-0.8,-2.5) node {$X_1$};

\draw [dotted] (0,-1.4) ellipse (2cm and 4cm);
\draw (0,2) node {$X$};

\foreach \pos/\name in  {{(3,-3.7)/f}, {(3,-2.7)/g}, {(3,-1.7)/h}, {(3,-0.7)/i}, {(3,0.3)/j}, {(3,1.3)/k}}
        \node[green] (\name) at \pos {};

\draw [dotted] (3,-1.2) ellipse (0.9cm and 3.6cm);
\draw (3,2) node {$Y$};
\end{tikzpicture}\quad

\end{minipage}
\begin{minipage}[t]{0.5\textwidth}

\begin{tikzpicture}[style=thick]
\foreach \pos/\name in  {{(18:1.5cm)/a}, {(90:1.5cm)/b},  {(162:1.5cm)/c}, {(234:1.5cm)/d}, {(306:1.5cm)/e}}
        \node[green] (\name) at \pos {};
\draw (a) -- (b) -- (c) -- (d) -- (e) -- (a);

\foreach \pos/\name in  {{(-1,-4)/n}, {(0,-2.5)/l}, {(1,-4)/m}}
        \node[green] (\name) at \pos {};
\draw (n) -- (l) -- (m) -- (n);

\foreach \pos/\name in  {{(3,-3.7)/f}, {(3,-2.7)/g}, {(3,-1.7)/h}, {(3,-0.7)/i}, {(3,0.3)/j}, {(3,1.3)/k}}
        \node[green] (\name) at \pos {};

\foreach \a in {f, g, h, i, j, k}
	\foreach \b in {a, b, c, d, e, l, m, n}
		\draw (\a) -- (\b);

\draw [dotted] (0.9,-1.4) ellipse (3.6cm and 4cm);
\draw (0.9,2) node {$Q_4$};

\end{tikzpicture}

\end{minipage}

\caption[$Q_4$: 8-regular graph with no internal partition (right) is composed from components (left)]{$Q_4$: 8-regular graph with no internal partition (right) is composed from components (left)}
\label{q4}
\end{figure}
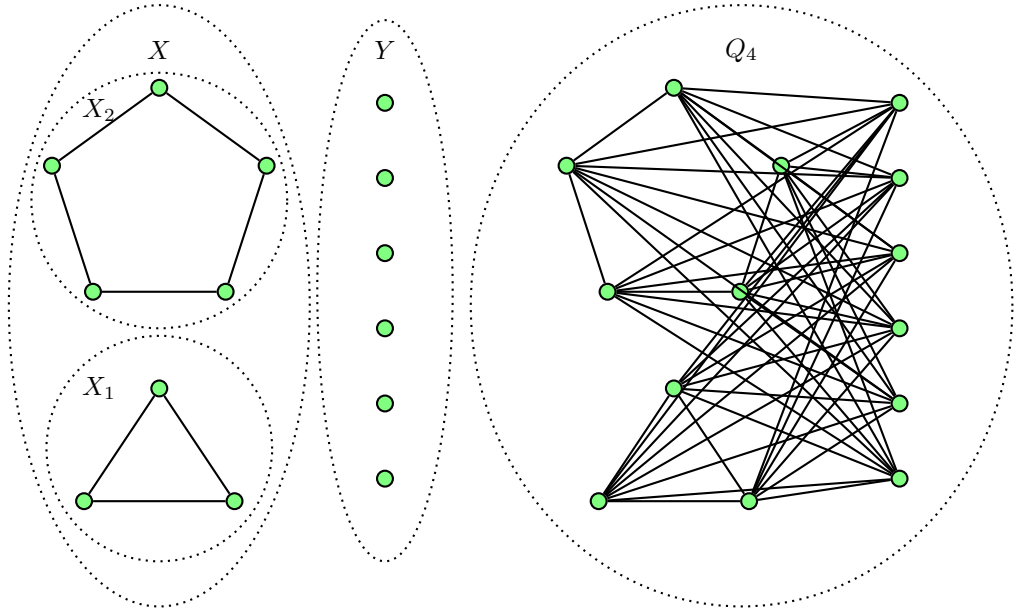

We construct a class of graphs without an internal partition, in which both $d$ and $n-d$ are unbounded:

Given an integer $m > 2$, construct the graph $Q_m$ as follows (see Figure~\ref{q4}):
\begin{enumerate}
\item Start with a $X_1 := K_{m-1}$ component.
\item Let $X_2$ be an $(m+1)$-vertex, $(m-2)$-regular graph, and let $X$ be the graph with components $X_1, X_2$.
\item Let $Y := \bar{K}_{m+2}$ (i.e. $Y$ has $m+2$ isolated vertices).
\item Finally $Q_m$ is attained by adding to $X, Y$ the complete bipartite graph between $V(X)$ and $V(Y)$.
\end{enumerate}

$Q_m$ is $2m$-regular with $3m+2$ vertices. The first few such graphs are $Q_3 (n=11, d=6)$, $Q_4 (n=14, d=8)$, $Q_5 (n=17, d=10)$, \ldots.

\begin{proposition}
$Q_m$ has no internal partition.
\end{proposition}

\begin{proof}
Suppose to the contrary that $(A,B)$ is an internal partition of $Q_m$ with $|A|=a$ and $|B|=b$. In the complementary graph $\bar{Q}_m$, the set $Y$ are the vertices of a $K_{m+2}$ component. In the partition $(A,B)$ of $\bar{Q}_m$, each vertex in $A$ (resp. $B$) has outdegree at least $b-m$ (resp. $a-m$). The only way to partition $K_{m+2}$ to meet these requirements is to have $a-m$ of its vertices in $A$ and the other $b-m$ vertices in $B$

Therefore $|V(X) \cap A| = |V(X) \cap B| = m$. Also $\forall{x \in (V(X) \cap A)}, d_{V(X) \cap A}(x) \geq m - (a-m) = 2m-a$, and $\forall{x \in (V(X) \cap B)}, d_{V(X) \cap B}(x) \geq m - (b-m) = 2m-b$. Therefore $(V(X) \cap A, V(X) \cap B)$ is a $q$-internal partition of $X$ for $q = \frac{2m-a}{m-2}$. Now $X_1$, being complete, has no $q$-internal partition for any $q$. Therefore its vertices are either all in $A$ or all in $B$. Say in $A$. Then $|V(X_2) \cap B| = m - |V(X_1)| = 1$, so there is a single $B$-vertex in the $X_2$ component, but a partition of a connected graph into a single vertex and its complement is not $q$-internal for any $q$. A contradiction.
\qed
\end{proof}

The reader will note that for all known examples $G$ of even-degree regular graphs with no internal partition, the complement $\bar G$ is disconnected. We do not know whether this is true in general, but we observe that if true, this implies $2d > n$. To the best of our knowledge, this may hold in general:

\begin{conjecture}
\label{conj2d}
For every even $d$, every $d$-regular graph with no internal partition has less than $2d$ vertices.
\end{conjecture}

We return to the problem of the existence of a $q$-internal partition for arbitrary regular graphs. There is a distinction between integral and non-integral partitions. Non-integral partitions are rarer than integral partitions, since every $q$-internal partition of a $d$-regular graph $G$ is also an integral $q'$-internal partition of $G$ for $q' = \lfloor qd \rfloor/d$ as well as for $q' = \lceil qd \rceil/d$. We make the following conjecture:

\begin{conjecture}
For every integer $d$ and $1>q>0$ such that either (i) $q = \frac{1}{2}$ or (ii) $qd$ is an integer, there is an integer $\mu$ such that every $d$-regular graph of order $\geq \mu$ has a $q$-internal partition.
\end{conjecture}

As already noted, $\mu=8$ for $d=3, q=\frac{1}{2}$.
Numerical experiments suggest that for $q=\frac{1}{2}$ and $d=5, 7$ there holds $\mu=18, 26$ respectively.

In fact, the following stronger statement appears to be true: There exists an integer $\mu'$ that depends only on $d_{min}(G),d_{max}(G)$ and on $q$ such that every graph $G=(V,E)$ with order at least $\mu'$ has a $q$-internal partition if (i) $q = \frac{1}{2}$ or (ii) $qd_G(v)$ is a positive integer for all $v \in V$.

For other values of $q$ (i.e. with non-integral values of $qd$ other than $q = \frac{1}{2}$), we make no guesses. We note that, for example, a connected graph cannot have a $q$-internal partition for $0 < q < \frac{1}{d}$. On the other hand, for $\frac{1}{d} < q < \frac{2}{d}$, a shortest cycle and its complement often yield a $q$-internal partition (e.g., when the girth is $\ge 5$).

Although the above conjecture remains open, the following theorem shows that every incomplete graph has an integral $q$-internal partition for {\em some} $q$. Moreover, for $d$ fixed and growing $n$ the number of such distinct partitions tends to $\infty$.

\begin{theorem}
\label{existq}
A $d$-regular graph $G$ of order $n>d+1$ has a $q$-internal partition $(A,B)$ for some $q \in (0,1)$ with $qd$ an integer. Such partitions exist for at least $\frac{n-d-1}{d}$ different values of $|A|$.
\end{theorem}

\begin{proof}
$\bar{G}$ is ($n-d-1$)-regular. Select $r \in (0,1)$ such that $r(n-d-1)$ is not an integer. This is always possible since $n-d-1 \ne 0$. By Proposition \ref{coexist} $\bar{G}$ has an $(1-r)$-external partition $(A,B)$.

In this partition of $\bar{G}$, $\forall x \in A, d_{\bar{A}}(x) < r(n-d-1)$. The inequality is strict since $r(n-d-1)$ is not an integer. Similarly $\forall x \in B, d_{\bar{B}}(x) < (1-r)(n-d-1)$.

Considering the partition $(A,B)$ in $G$, we have $\forall x \in A, d_A(x) > |A| - 1 - r(n-d-1)$, and $\forall x \in B, d_B(x) > |B| - 1 - (1-r)(n-d-1)$. Therefore
\begin{align}
\label{eq1}
\forall x \in A, & d_A(x) \geq |A| - 1 - \lfloor{r(n-d-1)}\rfloor =  |A|  - \lceil{r(n-d-1)}\rceil \\
\label{eq2}
\forall x \in B, & d_B(x) \geq |B| - 1 - \lfloor{(1-r)(n-d-1)}\rfloor =  |B|  - \lceil{(1-r)(n-d-1)}\rceil
\end{align}

Set $q = (|A|  - \lceil{r(n-d-1)}\rceil) / d$. By \eqref{eq1} the minimal indegree of $A$ is suitable for a $q$-internal partition. As for $B$, note that $\lfloor{(1-r)(n-d-1)}\rfloor +  \lceil{r(n-d-1)}\rceil = n-d-1$. So:
\begin{equation}
|B| - 1 - \lfloor{(1-r)(n-d-1)}\rfloor = n - |A| - 1 - (n-d-1) + \lceil{r(n-d-1)}\rceil = (1 - q)d
\end{equation}

Therefore by \eqref{eq2} the minimal indegree of $B$ is also suitable, and $(A,B)$ is a $q$-internal partition.

From \eqref{eq1} we see that:
\begin{equation}
\label{boundA}
 \lceil{r(n-d-1)}\rceil \leq |A| \leq  \lceil{r(n-d-1)}\rceil  + d
\end{equation}

So for any given $r$, $|A|$ has a range of at most $d$. Since $\lceil{r(n-d-1)}\rceil$ can take on $n-d-1$ values, $|A|$ takes on at least $\frac{n-d-1}{d}$ different values. The number of distinct $q$-internal partitions is at least as many.
\qed
\end{proof}

For $d$ fixed there are just $d-1$ values of $q \in (0,1)$ for which $qd$ is integral. By Theorem \ref{existq} every $d$-regular graph has $\Omega(n)$ distinct integral $q$-internal partitions. While this does not prove the existence of a $q$-internal partition for any {\em specific} $q$, it suggests that this becomes more likely as $n$ grows.

From Theorem \ref{existq} we derive an efficient algorithm that generates integral $q$-internal partitions for many and, for $n \gg d$, often {\em all} possible values of $q$:
\begin{algorithm}
\label{generate}
Given a $d$-regular graph $G=(V,E)$ with $n = |V|$:
\begin{enumerate}
\item Set $A \leftarrow \emptyset, B \leftarrow V$.
\item For $p = 1, \ldots, n - d - 1$
\begin{enumerate}
\item Repeat while $\exists{x \in B}, d_{\bar{A}}(x) < p$ or $\exists {x \in A}, d_{\bar{B}}(x) < n - d - p$
\begin{enumerate}
\item If $x \in A$ set $A \leftarrow A \setminus \{x\}, B \leftarrow B \cup \{x\}$ 
\item else set  $A \leftarrow A \cup \{x\}, B \leftarrow B \setminus \{x\}$
\end{enumerate}
\item Set $A_p \leftarrow A, B_p \leftarrow B$
\end{enumerate}
\end{enumerate}
\end{algorithm}

This algorithm generates the partitions $(A_p,B_p), p \in [n-d-1]$ of $\bar{G}$ each of which is $q$-external for $q = p/(n-d-1)$, by greedily moving vertices. When $p > 1$, the starting point for $(A_p,B_p)$ is $(A_{p-1},B_{p-1})$.

From Theorem \ref{existq} and its proof, $(A_p,B_p)$ is also a $q$-internal partition of $G$ for $qd = |A_p| - p$. Note that $A_1$ is a maximal independent set in $\bar{G}$, and so is $B_{n-d-1}$. Now when $n \gg  d$ the size of a maximal independent set is 2. Therefore, $|A_1| = 2$, $|A_{n-d-1}| = n-2$ and so $(A_1,B_1)$ is a $\frac{1}{d}$-internal partition of $G$ and $(A_{n-d-1},B_{n-d-1})$ is a $\frac{d-1}{d}$-internal partition of $G$.

Additionally from \eqref{boundA} $p \leq |A_p| \leq p+d$, so $|A_p|$ generally grows from $2$ to $n-2$ as $p$ grows from $1$ to $n-d-1$. The {\em average} of $|A_p| - |A_{p-1}|$ is $(n-4)/(n-d-2) \simeq 1$. Now since $(A_p,B_p)$ is a $q$-internal partition of $G$ for $q = \frac{|A_p| - p}{d}$, if it turns out that for all $p \in [n-d-2]$, $|A_{p+1}| - |A_p| < 3$, the algorithm generates all possible integral $q$-internal partitions of $G$.

Conversely, if for some graph $G$, some integral $q$-internal partition does not exist, then any sequence of partitions  $(A_p,B_p), p \in [n-d-1]$, whether generated by Algorithm \ref{generate} or by any other means, will exhibit a gap  $|A_p| - |A_{p-1}| \geq 3$ for some $p > 1$. For example, considering the graph $K_{3,3,3}$ (Figure \ref{d=n-3 example}) shown not to have an internal partition: $n-d-1 = 2$ and $|A_1| = 3, |A_2| = 6$.

\end{document}